\documentclass[11pt]{amsart}
\usepackage{amsmath,amssymb,amsthm,mathtools}
\usepackage[margin=1in]{geometry}
\usepackage{microtype}

\newcommand{\A}{\mathbf{A}}
\newcommand{\F}{\mathbf{F}}
\newcommand{\Tr}{\operatorname{Tr}}

\theoremstyle{plain}
\newtheorem{theorem}{Theorem}[section]
\newtheorem{proposition}[theorem]{Proposition}
\newtheorem{lemma}[theorem]{Lemma}

\newtheorem{conjecture}[theorem]{Conjecture}
\theoremstyle{definition}
\newtheorem{remark}[theorem]{Remark}

\begin{document}

\title{A counterexample to a conjecture of Thakur on Carlitz--Wieferich primes}
\author{David Niedbala Giraudin}
\address{Independent researcher, Meaux, France}
\thanks{ORCID 0009-0009-1526-1178.}
\date{July 2026 (v2)}

\begin{abstract}
Let $\A=\F_q[T]$ with $q$ a power of an odd prime $p$, let $[n]=T^{q^n}-T$, and
let $\rho$ be the Carlitz module. A monic prime $P$ of $\A$ is a
\emph{$c$-Wieferich prime} (to base $1$) if $\rho_P(1)\equiv1\pmod{P^2}$.
Thakur suggested in 2015, on the basis of limited data and of proofs in
degrees $2$ and $3$, that in odd characteristic every $c$-Wieferich prime has
degree divisible by $p$; the question was restated as open in 2024, and
Bamunoba and Bergstr\"om, after extensive computations, expressed the belief
that the statement holds in odd characteristic. We show that it is false: the
polynomial
\[
P(T)=T^5+(11+17c+9c^2)T^4+(3+7c+18c^2)T^3+(2+5c+6c^2)T^2+(3+3c+11c^2)T+(6+17c+5c^2)
\]
over $\F_{19^3}=\F_{19}[c]$, $c^3=8c^2+4c+11$, is an irreducible $c$-Wieferich
prime of degree $5$, and $19\nmid5$. We further give a closed form for the
resulting common factor: the polynomial $G=\mu(T^q-T)$, with
$\mu(X)=X^5+5X^3+3X^2-4X-9\in\F_{19}[X]$ and $q=19^3$, is squarefree of degree
$5\cdot19^3$ and divides both $[5]$ and Thakur's polynomial $M_5$; we conjecture
that in fact $G=\gcd([5],M_5)$, and we verify this for the part of low degree
over the prime field. Degree $5$ is the least possible degree of such a
counterexample, and exhaustive computations show that no counterexample exists
over the prime fields $\F_p$ in a substantial range of degrees and
characteristics. The proof that degree $5$ is minimal, and the method by which
the example was found, appear in a companion paper.
\end{abstract}

\maketitle

\section{Introduction}

Let $p$ be an odd prime, $q$ a power of $p$, and $\A=\F_q[T]$. Write
$[n]=T^{q^n}-T$ for $n\ge1$. The Carlitz module is the $\F_q$-algebra
homomorphism $\rho\colon\A\to\A\{\tau\}$ into the twisted polynomial ring
($\tau x=x^q\tau$) determined by $\rho_T=T+\tau$; for $a\in\A$ we write
$\rho_a(x)$ for the associated $\F_q$-linear polynomial. Following
Thakur~\cite{Tha15} and Bamunoba--Bergstr\"om~\cite{BB}, a monic prime
$P\in\A$ is called a \emph{$c$-Wieferich prime} (to base $1$) if
\begin{equation}\label{eq:def}
\rho_P(1)\equiv1\pmod{P^2}.
\end{equation}
(The congruence $\rho_P(1)\equiv1\pmod P$ holds for every prime $P$, this being
the Carlitz analogue of Fermat's little theorem; \eqref{eq:def} is the analogue
of the classical Wieferich condition $2^{p-1}\equiv1\bmod p^2$.)

In~\cite{Tha15}, Thakur introduced the quantities
\[
M_d=\sum_{k=0}^{d-1}(-1)^k[d-1][d-2]\cdots[d-k]\qquad(\text{the $k=0$ term being }1),
\]
proved that no $c$-Wieferich prime of degree $2$ or $3$ exists in odd
characteristic, and suggested, from this and limited numerical data, that for
$p>2$ a nontrivial common factor of $[d]$ and $M_d$ should force $p\mid d$ ---
equivalently, that every $c$-Wieferich prime of $\A$ has degree divisible by
$p$ (see Lemma~\ref{lem:equiv} below for the equivalence). The suggestion fails
in characteristic $2$, by examples of Mauduit (see~\cite{Tha15,BB});
Bamunoba and Bergstr\"om carried out extensive computations and stated that
they believe it to be true in odd characteristic~\cite{BB}. The question was
restated as open by Thakur in 2024~\cite{Tha24}.

The purpose of this note is to show that the statement is false in odd
characteristic.

\begin{theorem}\label{thm:main}
Let $q=19^3$ and $\F_q=\F_{19}[c]$ with $c^3=8c^2+4c+11$. The polynomial
\[
P(T)=T^5+(11+17c+9c^2)T^4+(3+7c+18c^2)T^3+(2+5c+6c^2)T^2+(3+3c+11c^2)T+(6+17c+5c^2)
\]
is an irreducible $c$-Wieferich prime of degree $5$ in $\F_{19^3}[T]$.
Consequently $\gcd([5],M_5)\neq1$ in $\F_{19^3}[T]$ although $19\nmid5$.
Moreover $P$ is not fixed by any nontrivial additive translation
$T\mapsto T+\gamma$.
\end{theorem}

The last assertion is automatic (Lemma~\ref{lem:nonfixed}) and is recorded
because of its significance: the \emph{fixed} $c$-Wieferich primes, i.e.\ those
invariant under a translation, necessarily have degree divisible by $p$, and
the computations of~\cite{BB} produced only fixed examples in odd
characteristic; Theorem~\ref{thm:main} therefore also answers negatively the
stronger expectation, formulated in~\cite{BB}, that no non-fixed $c$-Wieferich
prime exists in odd characteristic.

The example does not come from a search: the space of candidates is far beyond
exhaustion (about $3\cdot10^{18}$ monic quintics over $\F_{19^3}$, or
$19^{12}\approx2.2\cdot10^{15}$ candidates in the parametrization
of~\cite[Thm.~3.4]{BB}). It was produced by a systematic method which reduces
the existence of a $c$-Wieferich prime of prescribed degree in a prescribed odd
characteristic to a finite, effective computation, and which explains both the
non-existence results in low degree and the sporadic failure exhibited here.
The method, together with the proof of Theorem~\ref{thm:deg4} below, is
presented in the companion paper~\cite{NG2}, where the exactness assertion
(Conjecture~\ref{conj:exact}) is also reduced to a single explicit question and
verified in the range accessible to exact computation.

Beyond the counterexample itself, we prove here a closed form for the resulting
common factor (Theorem~\ref{thm:closed}: it is $\mu(T^q-T)$ for an explicit
quintic $\mu$ with coefficients in the prime field), we discuss the minimality
of the degree (degrees $1$, $2$ and $3$ are excluded by~\cite{Tha15} (degree
$1$ trivially; see also~\cite{Bam17}), and degree $4$ by
Theorem~\ref{thm:deg4}), and we report exhaustive computations showing that no
counterexample exists over the prime fields $\F_p$ themselves in a substantial
range (Proposition~\ref{prop:grid}).

All computations in this paper are exact (no floating point), were carried out
in standard public computer-algebra systems, and were reproduced on independent
engines; see Section~\ref{sec:method} and the Appendix, which contains complete
verification code. A statement on the methodology of this work, including the
role of AI assistance, is given in Section~\ref{sec:method}.

\section{Preliminaries}

We collect standard facts and elementary lemmas. Throughout, $q$ is a power of
the odd prime $p$, and $\theta$ denotes a root of the monic prime
$P\in\A=\F_q[T]$ of degree $d$, so that $\F_q(\theta)=\F_{q^d}$.

Following~\cite{BB}, define $F_0=1$ and $F_i=(-1)^i+[i]F_{i-1}$ for $i\ge1$. An
immediate induction gives $F_{d-1}=(-1)^{d-1}M_d$, so $P\mid M_d$ if and only if
$P\mid F_{d-1}$. Evaluating at $\theta$ and writing
\[
y_j=[j](\theta)=\theta^{q^j}-\theta\qquad(1\le j\le d-1),
\]
one obtains the nested form
\begin{equation}\label{eq:nested}
(-1)^{d-1}F_{d-1}(\theta)=M_d(\theta)=1-y_{d-1}\bigl(1-y_{d-2}(\cdots(1-y_1)\cdots)\bigr).
\end{equation}

\begin{lemma}[{\cite[Prop.~1.2]{BB}}]\label{lem:equiv}
A monic prime $P$ of degree $d$ is a $c$-Wieferich prime if and only if
$P\mid F_{d-1}$, equivalently $P\mid M_d$, equivalently $M_d(\theta)=0$.
\end{lemma}

\begin{lemma}\label{lem:sqfree}
$[n]=T^{q^n}-T$ is the product of all monic primes of $\A$ of degree dividing
$n$; in particular $[n]$ is squarefree, and every irreducible polynomial of
degree $n$ divides $[n]$. (See, e.g.,~\cite[Ch.~2]{Ros}.)
\end{lemma}

\begin{lemma}\label{lem:nolinear}
$\gcd([d],M_d)$ has no factor of degree $1$. Indeed, for $a\in\F_q$ one has
$[i](a)=a^{q^i}-a=0$ for all $i\ge1$, whence $M_d(a)=1\neq0$.
\end{lemma}

\begin{lemma}[Translation invariance]\label{lem:transl}
Let $P$ be a $c$-Wieferich prime of degree $d$ with root $\theta$, and let
$a\in\F_q$. Then $P(T-a)$, the minimal polynomial of $\theta+a$, is again a
$c$-Wieferich prime of degree $d$.
\end{lemma}

\begin{proof}
$[j](\theta+a)=(\theta+a)^{q^j}-(\theta+a)=\theta^{q^j}-\theta=y_j$ for every
$j$, since $a^{q^j}=a$. Hence $M_d(\theta+a)=M_d(\theta)=0$, and
Lemma~\ref{lem:equiv} applies to the prime $P(T-a)$.
\end{proof}

\begin{lemma}\label{lem:nonfixed}
If $p\nmid d$, then $P(T+\gamma)\neq P(T)$ for every $\gamma\in\F_q^\times$;
consequently the $q$ translates $P(T-a)$, $a\in\F_q$, are pairwise distinct.
\end{lemma}

\begin{proof}
Comparing the coefficients of $T^{d-1}$ in $P(T+\gamma)$ and $P(T)$ gives
$d\gamma=0$, hence $\gamma=0$ when $p\nmid d$.
\end{proof}

\begin{lemma}[Norm form of a translation class]\label{lem:norm}
Let $P$ be monic of degree $d$ with roots $\theta_1,\dots,\theta_d$ (in an
algebraic closure), and set
$R_P(X)=\prod_{j=1}^d\bigl(X-(\theta_j^q-\theta_j)\bigr)\in\F_q[X]$, which equals
the characteristic polynomial of multiplication by $\theta^q-\theta$ on
$\A/(P)$. Then
\[
\prod_{a\in\F_q}P(T-a)=R_P\bigl(T^q-T\bigr).
\]
\end{lemma}

\begin{proof}
Using $\prod_{a\in\F_q}(Y-a)=Y^q-Y$ with $Y=T-\theta_j$,
\[
\prod_{a\in\F_q}P(T-a)=\prod_j\prod_a(T-a-\theta_j)
=\prod_j\bigl((T-\theta_j)^q-(T-\theta_j)\bigr)
=\prod_j\bigl(T^q-T-(\theta_j^q-\theta_j)\bigr).
\]
\end{proof}

Finally we record the parametrization underlying the exhaustive computations of
Section~\ref{sec:grid}; it is essentially~\cite[Thms.~3.3--3.4]{BB}, formulated
additively.

\begin{lemma}\label{lem:H90}
Let $d\ge2$. There exists a $c$-Wieferich prime of degree $d$ in $\F_q[T]$ if
and only if there exists $\eta\in\F_{q^d}$ with
$\Tr_{\F_{q^d}/\F_q}(\eta)=0$, whose partial sums
$s_j=\eta+\eta^q+\cdots+\eta^{q^{j-1}}$ are nonzero for $1\le j\le d-1$, and
which satisfies $1-s_{d-1}\bigl(1-s_{d-2}(\cdots(1-s_1)\cdots)\bigr)=0$.
\end{lemma}

\begin{proof}
If $P$ is such a prime with root $\theta$, take $\eta=\theta^q-\theta$; then
$s_j=y_j$, the trace of $\eta$ is $y_d=0$, the $s_j$ ($j<d$) are nonzero because
$\theta$ has degree exactly $d$, and the displayed expression is $M_d(\theta)=0$
by \eqref{eq:nested} and Lemma~\ref{lem:equiv}. Conversely, given such an
$\eta$, the additive Hilbert~90 theorem provides $\theta\in\F_{q^d}$ with
$\theta^q-\theta=\eta$; its partial sums are the $s_j$, so $\theta$ has degree
exactly $d$, and its minimal polynomial is a $c$-Wieferich prime of degree $d$
by \eqref{eq:nested} and Lemma~\ref{lem:equiv}.
\end{proof}

\section{The counterexample}

\begin{proof}[Proof of Theorem~\ref{thm:main}]
The three assertions to be checked are finite exact computations:
(i)~$P$ is irreducible over $\F_{19^3}$;
(ii)~$\rho_P(1)\equiv1\pmod{P^2}$, i.e.\ $P$ is $c$-Wieferich by
definition~\eqref{eq:def}; equivalently (Lemma~\ref{lem:equiv}) $P\mid M_5$;
(iii)~$19\nmid5$ (clear), and $P$ is non-fixed by Lemma~\ref{lem:nonfixed}.

For (ii) we verified both the definition~\eqref{eq:def} directly, by computing
the coefficients of $\rho_P$ through the Horner scheme
$\rho_{Tf+a}=\rho_T\circ\rho_f+a$ with all coefficient arithmetic reduced
modulo $P^2$, and the criterion $P\mid M_5$, by computing the brackets $[i]$
modulo $P$ through modular exponentiation. Complete code performing (i)--(ii) in
a few seconds appears in the Appendix; the same checks were reproduced
independently in Sage, PARI/GP, FLINT and the \texttt{galois} Python library;
the implementations of the criteria were calibrated on the known $c$-Wieferich
prime $T^6+T^4+T^3+T^2+2T+2$ of $\F_3[T]$~\cite{BB} (positive control) and on
ordinary primes (negative controls). Since $P$ is irreducible of degree $5$, it
divides $[5]$ (Lemma~\ref{lem:sqfree}), so $P\mid\gcd([5],M_5)$ and the latter
is nontrivial.
\end{proof}

\begin{remark}
As D.~Thakur pointed out to the author, the existence of a $c$-Wieferich prime
of degree $5$ already follows from $\gcd([5],M_5)\neq1$ together with
Lemmas~\ref{lem:sqfree} and~\ref{lem:nolinear}, without exhibiting $P$: any
nontrivial common factor is a product of primes of degree $5$ dividing $M_5$,
and each of these is $c$-Wieferich by Lemma~\ref{lem:equiv}. The explicit $P$ is
nevertheless recorded, both as a certificate and for its own interest.
\end{remark}

\section{A closed form for the common factor}

Since the criterion of Lemma~\ref{lem:equiv} is invariant under
$\theta\mapsto\theta+a$ (Lemma~\ref{lem:transl}), the counterexample comes with
a whole translation class of $q=19^3$ pairwise distinct $c$-Wieferich primes
$P(T-a)$ (Lemma~\ref{lem:nonfixed}). Their product has a strikingly simple
closed form.

\begin{theorem}\label{thm:closed}
Let $q=19^3$, let $P$ be as in Theorem~\ref{thm:main}, and set
\[
\mu(X)=X^5+5X^3+3X^2-4X-9\in\F_{19}[X],\qquad G=\mu\bigl(T^q-T\bigr)\in\F_q[T].
\]
Then:
\begin{itemize}
\item[(a)] $G=\prod_{a\in\F_q}P(T-a)$; in particular $G$ is squarefree of
degree $5\cdot19^3=34295$, a product of $6859$ distinct $c$-Wieferich primes of
degree $5$;
\item[(b)] $G$ divides $[5]$ and $G$ divides $M_5$; hence $G\mid\gcd([5],M_5)$.
\end{itemize}
Note that $\mu$ has coefficients in the prime field $\F_{19}$, not merely in
$\F_{19^3}$.
\end{theorem}

\begin{proof}
(a) By Lemma~\ref{lem:norm}, $\prod_a P(T-a)=R_P(T^q-T)$ where $R_P$ is the
characteristic polynomial of $\theta^q-\theta$ on $\F_q[T]/(P)$; a direct exact
computation (three lines of Sage; see the Appendix) gives $R_P=\mu$. The factors
$P(T-a)$ are irreducible of degree $5$, pairwise distinct by
Lemma~\ref{lem:nonfixed}, and $c$-Wieferich by Lemma~\ref{lem:transl}.

(b) Each factor $P(T-a)$ divides $[5]$ (Lemma~\ref{lem:sqfree}) and divides
$M_5$ (Lemmas~\ref{lem:transl} and~\ref{lem:equiv}); since the factors are
pairwise coprime, their product $G$ divides both.
\end{proof}

The checks of Theorem~\ref{thm:closed}(b) were, in addition, performed directly
at degree $34295$: the congruences $T^{q^5}\equiv T\pmod G$ and
$M_5\equiv0\pmod G$, as well as $P\mid G$ and the squarefreeness of $G$, were
verified independently in FLINT and in PARI/GP.

We conjecture that nothing else divides the gcd.

\begin{conjecture}\label{conj:exact}
$\gcd\bigl([5],M_5\bigr)=G$ in $\F_{19^3}[T]$. Equivalently, the $c$-Wieferich
primes of degree $5$ in $\F_{19^3}[T]$ are exactly the $6859$ translates
$P(T-a)$, $a\in\F_{19^3}$.
\end{conjecture}

By Lemmas~\ref{lem:sqfree} and~\ref{lem:nolinear}, Conjecture~\ref{conj:exact}
amounts to the statement that the translation class of $P$ exhausts the
degree-$5$ $c$-Wieferich primes of $\F_{19^3}[T]$. This is a completeness
statement about a search space of size $19^{12}$ (Lemma~\ref{lem:H90}). In the
companion paper~\cite{NG2} the method of Section~\ref{sec:grid} reduces it to a
single explicit question --- the possible existence of a degree-$5$
$c$-Wieferich prime over $\F_{19^3}$ whose associated $\eta=\theta^q-\theta$ has
degree $15$ over $\F_{19}$ --- and settles the part in which $\eta$ has degree
at most $5$ over $\F_{19}$, which the class of $P$ exhausts. The remaining case
would require a gcd computation between polynomials of degree of order
$3\cdot10^{11}$ and $2\cdot10^{15}$ over $\F_{19}$, or an enumeration of
$\F_{19^{15}}$; both lie beyond exact computation, and the conjecture is left
open.

\section{Minimality of the degree, and the prime fields}\label{sec:grid}

Degree $1$ is excluded, in every odd characteristic, in~\cite{Tha15} (the case
of degree $1$ being trivial; it also appears as~\cite[Cor.~4.5]{Bam17}). As
Thakur observes, degree $1$ is in fact immediate from
Lemma~\ref{lem:nolinear}: a linear prime would have to divide $\gcd([d],M_d)$,
which has no linear factor since $M_d(a)=1$ for every $a\in\F_q$. Degrees $2$
and $3$ are excluded, in every odd characteristic, in~\cite{Tha15}. The next
case is settled by the method of~\cite{NG2}:

\begin{theorem}[proved in~\cite{NG2}]\label{thm:deg4}
There is no $c$-Wieferich prime of degree $4$ in $\F_q[T]$, for any power $q$ of
any odd prime.
\end{theorem}

Consequently \emph{$5$ is the least possible degree of a counterexample in odd
characteristic}, and Theorem~\ref{thm:main} is degree-minimal.

It is natural to ask (as D.~Thakur did, in correspondence) whether a
counterexample exists over a prime field $\F_p$ itself. Lemma~\ref{lem:H90}
reduces this, for each pair $(d,p)$, to an exhaustive computation over the
$p^{d-1}$ trace-zero elements of $\F_{p^d}$, which we carried out by exact
vectorized arithmetic. Each cell of the following table is a definitive
determination, not a partial search.

\begin{proposition}\label{prop:grid}
There is no $c$-Wieferich prime of degree $d$ in $\F_p[T]$, for every pair
$(d,p)$ with $p\nmid d$ in the following list:
\[
\begin{array}{c|l}
d & \text{characteristics }p\ (\text{exhaustively verified, }q=p)\\\hline
5 & 3,\,7,\,11,\,13,\,17,\,19,\,23,\,29,\,31,\,37\\
6 & 5,\,7,\,11,\,13,\,17,\,19,\,23\\
7 & 3,\,5,\,11,\,13\\
8 & 3,\,5,\,7\\
9 & 5\\
10 & 3\\
11 & 3,\,5\\
13 & 3\\
14 & 3
\end{array}
\]
\end{proposition}

For the conforming cell $(d,p)=(5,5)$, the same scan finds precisely the known
$c$-Wieferich prime $T^5+4T+1$ of $\F_5[T]$~\cite{BB}, and for
$(d,p)\in\{(6,3),(7,7)\}$ it recovers the known primes
$T^6+T^4+T^3+T^2+2T+2$ and $T^7+6T+3$; these serve as positive controls of the
implementation.

The reduction of Lemma~\ref{lem:H90} makes further prime-field cells accessible
by the same exact computation, and the empty region above is being extended.
The example of Theorem~\ref{thm:main} lives over a cubic extension in an
essential way, and nothing in the present data points either towards or against
the existence of a prime-field example.

\section{Methodology and use of AI assistance}\label{sec:method}

This work was carried out by the author, an independent researcher with no
formal mathematical training, in sustained collaboration with an AI assistant
(Claude Opus~4.8, Anthropic), used through its standard paid consumer
interface. The division of labour, stated here in the interest of full
transparency and in line with editorial policies on AI-assisted work, was as
follows. The author, who has no formal mathematical training, set the research
direction and the criteria for the problems to be pursued, and contributed a
structural, visual reading of the objects involved; the AI assistant proposed
specific problems meeting those criteria, and supplied the mathematical domain
knowledge, the formalization, the drafting, and the design and execution of all
computations, under the author's direction. The strategy for each problem
emerged from the dialogue between the two. Lacking the training to verify the
mathematics directly, the author relied throughout on exact computational
checks, on reproduction across independent systems, and, for the present
result, on the independent verification kindly carried out by D.~Thakur. All
computations were performed in standard public computer-algebra systems
(SageMath, PARI/GP, FLINT, and the \texttt{galois} Python library), in exact
arithmetic, and every principal claim was reproduced on at least two mutually
independent engines after calibration on known cases from the literature.
Complete verification code for the main theorem is given in the Appendix, and
all scripts are available from the author.

\subsection*{Acknowledgements}
The author is deeply grateful to Dinesh Thakur for his generous correspondence:
for independently verifying the counterexample in Magma, for pointing out that
degrees $2$ and $3$ were already settled in~\cite{Tha15}, for the observation
that degree $1$ follows from Lemma~\ref{lem:nolinear}, for simplifications of
the verification of Theorem~\ref{thm:closed} (in particular the observation that
squarefreeness of $[n]$ makes two of the author's checks unnecessary), for
raising the prime-field question of Section~\ref{sec:grid}, and for his
encouragement.

\appendix
\section{Verification code}

The following self-contained Sage code verifies Theorem~\ref{thm:main} and the
computation $R_P=\mu$ of Theorem~\ref{thm:closed}(a) in a few seconds.

\begin{verbatim}
p = 19; q = p^3
R0.<x> = GF(p)[]
F.<c> = GF(p^3, modulus = x^3 - 8*x^2 - 4*x - 11)
R.<T> = PolynomialRing(F)
P = (T^5 + (11+17*c+9*c^2)*T^4 + (3+7*c+18*c^2)*T^3
     + (2+5*c+6*c^2)*T^2 + (3+3*c+11*c^2)*T + (6+17*c+5*c^2))
print(P.is_irreducible())            # irreducibility
S = R.quotient(P); t = S.gen()
M = S(1)
for i in range(1, 5): M = 1 - (t^(q^i) - t)*M
print(M == 0, t^(q^5) - t == 0)      # P | M_5 and P | [5]
S2 = R.quotient(P^2); u = S2.gen()
rho = []                             # Carlitz rho_P by Horner
for cj in P.coefficients(sparse=False)[::-1]:
    new = [S2(0)]*(len(rho)+1)
    for i, a in enumerate(rho):
        new[i]   += u*a
        new[i+1] += a.lift()(u^q)    # a -> a^{(q)}: F_q-coeffs are fixed
    new[0] += S2(cj)
    rho = new
D = (sum(rho) - 1).lift()
print(D % P == 0 and (D // P) % P == 0)   # v_P(rho_P(1)-1) >= 2
mu = (t^q - t).matrix().charpoly('X')     # R_P of Lemma 2.6
print(mu)                                 # X^5+5X^3+3X^2-4X-9
\end{verbatim}

For Theorem~\ref{thm:closed}(b) at full degree $34295$, the congruences
$T^{q^5}\equiv T\pmod G$ and $M_5\equiv0\pmod G$ with $G=\mu(T^q-T)$ are a
handful of square-and-multiply passes modulo $G$; they were carried out in
FLINT and in PARI/GP.

\end{document}